\documentclass{amsart}

\usepackage{hyperref}

\usepackage{amssymb}
\usepackage{amsmath}
\usepackage{amsthm}
\usepackage{enumerate}
\usepackage{graphicx}
\usepackage{tikz-cd}

\theoremstyle{plain}

\newtheorem{theorem}{Theorem}[section]

\newtheorem{proposition}[theorem]{Proposition}
\newtheorem{lemma}[theorem]{Lemma}
\newtheorem{corollary}[theorem]{Corollary}
\newtheorem{question}[theorem]{Question}

\theoremstyle{definition}
\newtheorem{definition}[theorem]{Definition}

\theoremstyle{remark}
\newtheorem*{remark}{Remark}

\newtheorem*{acknowledgments}{Acknowledgments}

\renewcommand {\tilde} {\widetilde}

\begin{document}

\title{On removable sets for holomorphic functions}

\date{March 6, 2015}

\author[M. Younsi]{Malik Younsi}
\thanks{Supported by NSERC}
\address{Department of Mathematics, Stony Brook University, Stony Brook, NY 11794-3651, United States.}
\email{malik.younsi@stonybrook.edu}


\begin{abstract}
We present a comprehensive survey on removability of compact plane sets with respect to various classes of holomorphic functions. We also discuss some applications and several open questions, some of which are new.
\end{abstract}

\maketitle

\section{Introduction}
A classical theorem generally attributed to Riemann asserts that every bounded holomorphic function on some open subset of the plane minus a point extends analytically to the whole open set. In other words, single points are removable for bounded holomorphic functions. A repeated application of Riemann's theorem obviously shows that any finite set is removable, so is any countable compact set, by a simple argument involving the Baire category theorem. What about uncountable compact sets? Is it possible to find a geometric characterization of those that are removable? This difficult problem dates back to Painlev\'e in 1888, who was the first one to investigate the properties of the compact plane sets that are removable for bounded holomorphic functions. Since then, the study of removable sets has been extended to several other classes of holomorphic functions. Understanding the properties of removable sets with respect to each of these classes has proven over the years to be of fundamental importance. Indeed, in many situations, the possibility of extending a function defined outside a compact set is somehow independent of the particular function but rather depends on which class it belongs to and on the geometric properties of the set.

The purpose of this article is to present a comprehensive survey on removability with respect to various classes of holomorphic functions, including proofs of some results that are probably well-known to experts but do not appear in the literature, as far as we know. Our main motivation comes from the fact that although there have been some excellent surveys on removable sets, each of them is either outdated or centered on only one type of removability. Moreover, some of the widely used removability results are essentially folklore theorems and our intention is to give rigorous proofs. Actually, as we will see, some statements have sometimes been taken for granted because they seem almost trivial at first glance, yet no rigorous and correct proof exists to the best of our knowledge. Another important motivation is that several people have expressed their interest in such a survey, particularly for the more modern notion of \textit{conformal removability}, which appears naturally in holomorphic dynamics. Indeed, it frequently happens that two dynamical systems are conjugated by some homeomorphism of the sphere which is (quasi)conformal outside the Julia set. If the latter is conformally removable, then the conjugation is in fact (quasi)conformal everywhere. Lastly, we would like to mention that the subjects treated in this survey reflect the author's own interests and are by no means exhaustive. There are several other notions of removability which we will not discuss, such as removability for bounded $K$-quasiregular mappings for example.

Let $E \subset \mathbb{C}$ be compact and let $\Omega:=\mathbb{C}_\infty \setminus E$ be the complement of $E$ with respect to the Riemann sphere. We shall be interested mainly in the following classes of holomorphic functions on $\Omega$ :

\begin{eqnarray*}
H^{\infty}(\Omega) &=& \{ \, f:\Omega \to \mathbb{C} \, \, \mbox{holomorphic and bounded} \, \}\\
A(\Omega) &=& \{ \, f:\mathbb{C}_\infty \to \mathbb{C} \, \, \mbox{continuous and holomorphic on} \, \, \Omega \, \}\\
S(\Omega) &=& \{ \, \mbox{conformal maps} \, \, f:\Omega \to \mathbb{C}_\infty \, \}\\
CH(\Omega) &=& \{\, \mbox{homeomorphisms} \, \, f:\mathbb{C}_\infty \to \mathbb{C}_\infty \, \, \mbox{which are conformal on} \, \, \Omega \, \}.\\
\end{eqnarray*}
Note that clearly, we have the inclusions
$$A(\Omega) \subset H^{\infty}(\Omega)$$
and
$$CH(\Omega) \subset S(\Omega).$$
Moreover, each of the above classes is monotonic in the sense that if $E_1 \subset E_2$, then the class corresponding to $\Omega_1 :=\mathbb{C}_\infty \setminus E_1$ is contained in the class corresponding to $\Omega_2:=\mathbb{C}_\infty \setminus E_2$.

\begin{definition}
Let $\mathcal{F}$ be one of the above classes of functions $H^{\infty},A,S$ or $CH$. We say that $E$ is $\mathcal{F}$-\textit{removable} if $\mathcal{F}(\Omega)=\mathcal{F}(\mathbb{C}_\infty)$, in other words, if every function in $\mathcal{F}(\Omega)$ is the restriction of an element of $\mathcal{F}(\mathbb{C}_\infty)$. More precisely, the compact set $E$ is $H^\infty,A,S,CH$-removable respectively if
\begin{itemize}
\item every bounded holomorphic function on $\Omega$ is constant;
\item every continuous function on $\mathbb{C}_\infty$ that is holomorphic on $\Omega$ is constant;
\item every conformal map on $\Omega$ is a M\"{o}bius transformation;
\item every homeomorphism of $\mathbb{C}_\infty$ onto itself that is conformal on $\Omega$ is a M\"{o}bius transformation.
\end{itemize}
\end{definition}

Note that in view of the preceding remarks, $H^{\infty}$-removable sets are $A$-removable and $S$-removable sets are $CH$-removable. Furthermore, any compact subset of a removable compact set is also removable, by monotonicity.

The rest of the paper is organized as follows. Sect. \ref{sec1} is about the study of $H^{\infty}$-removable sets. We discuss their main properties, particularly from the point of view of Hausdorff measure and dimension. This section also includes a brief introduction to analytic capacity and Tolsa's solution of Painlev\'e's problem. In Sect. \ref{sec2}, we present the main properties of $A$-removable compact sets, also from the point of view of Hausdorff measure, as well as a brief introduction to the analogue of analytic capacity in this new setting. Then, in Sect. \ref{sec3}, we first introduce some preliminaries on quasiconformal mappings and then proceed with a description of the properties of $S$-removable sets, including Ahlfors and Beurling's characterization based on the notion of zero absolute area. The last section, Sect. \ref{sec4}, deals with $CH$-removable compact sets. Among other things, we present Jones and Smirnov's geometric sufficient condition for $CH$-removability, as well as Bishop's construction of nonremovable sets of zero area. We also describe some applications of $CH$-removability to conformal welding and to the dynamics of complex quadratic polynomials. Finally, the section concludes with a discussion of several open questions.

\section{$H^{\infty}$-removable sets}
\label{sec1}
As mentioned in the introduction, the study of removable sets for bounded holomorphic functions has first been instigated by Painlev\'e \cite{PAI}, who raised the problem of finding a geometric characterization of $H^{\infty}$-removable compact sets. Moreover, Painlev\'e was the first one to observe that there is a close relationship between removability for $H^{\infty}$ and Hausdorff measure and dimension. More precisely, he proved that compact sets of Hausdorff dimension strictly less than one are removable. Unfortunately, the converse is false, as we will see later in this section. On the other hand, a well-known lemma of Frostman implies that compact sets of dimension strictly bigger than one are never removable. It follows that one is the critical Hausdorff dimension from the point of view of removability for bounded holomorphic functions. In this case, the situation is much more complicated and Painlev\'e's problem quickly turned out to be extremely difficult. As a matter of fact, it took more than a hundred years until a reasonable solution was obtained, thanks to the work of David, Tolsa and many others. A fundamental tool in the study of Painlev\'e's problem is the so-called \textit{analytic capacity}, an extremal problem introduced by Ahlfors \cite{AHL} in 1947 which, in some sense, measures the size of a compact set from the point of view of $H^{\infty}$-removability. We shall give a brief overview of the complete solution to Painlev\'e's problem at the end of this section, although the proofs will be omitted for the sake of conciseness. For more information on this vast subject, we refer the interested reader to the recent book of Tolsa \cite{TOL}. See also \cite{GAR} and \cite{DUD}.

\subsection{Main properties}

It follows from the aforementioned theorem of Riemann and Liouville's theorem that any single point is $H^{\infty}$-removable. The same argument obviously shows that every finite set is $H^{\infty}$-removable, so is any countable compact set, by simple argument using the Baire category theorem. In fact, we will see at the end of this subsection that any compact countable union of $H^{\infty}$-removable sets is also $H^{\infty}$-removable.

Intuitively, one expects $H^{\infty}$-removable sets to be small, at least in some sense. For instance, it is easy to see that if $E$ is $H^{\infty}$-removable, then its interior must be empty and its complement $\Omega=\mathbb{C}_\infty \setminus E$ must be connected. In fact, we have
\begin{proposition}
\label{TotDisc}
If $E$ is $H^{\infty}$-removable, then $E$ is totally disconnected.
\end{proposition}

\begin{proof}
If $F \subset E$ is a connected component of $E$ containing more than one point, then by the Riemann mapping theorem there exists a conformal map $f:\mathbb{C}_\infty \setminus F \to \mathbb{D}$. Clearly, $f$ is a nonconstant element of $H^{\infty}(\Omega)$.

\end{proof}
The following proposition shows that $H^{\infty}$-removability is a local property.
\begin{proposition}
\label{propextend}
The following are equivalent :

\begin{enumerate}[\rm(i)]
\item For any open set $U$ with $E \subset U$, every bounded holomorphic function on $U \setminus E$ extends analytically to an element of $H^{\infty}(U)$;
\item $E$ is $H^{\infty}$-removable.
\end{enumerate}
\end{proposition}

\begin{proof}
Clearly $(i)$ implies $(ii)$. The converse is a simple application of Cauchy's integral formula. Indeed, suppose that $E$ is removable and let $U$ be an open set containing $E$. Let $f$ be any bounded holomorphic function on $U \setminus E$ and fix $z \in U \setminus E$. Let $\Gamma_1$ be a cycle in $U \setminus (E \cup \{z\})$ with winding number one around $E \cup \{z\}$ and zero around $\mathbb{C} \setminus U$. Likewise, let $\Gamma_2$ be a cycle in $U \setminus (E \cup \{z\})$ with winding number one around $E$ and zero around $(\mathbb{C} \setminus U) \cup \{z\}$. Then by Cauchy's integral formula,
$$f(z)=\frac{1}{2\pi i}\int_{\Gamma_1} \frac{f(\zeta)}{\zeta-z} d\zeta - \frac{1}{2\pi i}\int_{\Gamma_2} \frac{f(\zeta)}{\zeta-z} d\zeta:=f_1(z)+f_2(z).$$
Note that by Cauchy's theorem, $f_1(z)$ and $f_2(z)$ do not depend on the precise cycles $\Gamma_1$ and $\Gamma_2$. It follows that $f_1$ and $f_2$ define holomorphic functions on $U$ and $\mathbb{C} \setminus E$ respectively, with $f=f_1+f_2$ on $U \setminus E$. Since $f_1$ and $f$ are bounded near $E$, the function $f_2$ is also bounded there and so is bounded everywhere outside $E$. But $E$ is removable and $f_2(\infty)=0$, hence $f_2$ is identically zero and therefore $f_1=f$ is the desired bounded analytic extension of $f$ to the whole open set $U$.
\end{proof}
A nice property of $H^{\infty}$-removable sets is that they are closed under unions.

\begin{proposition}
\label{union}
If $E,F$ are $H^{\infty}$-removable compact sets, then $E \cup F$ is also $H^{\infty}$-removable.
\end{proposition}
Note that if $E \cap F = \emptyset$, then the result is a direct consequence of Proposition \ref{propextend}. For the general case, we follow the proof in \cite[Proposition 1.18]{TOL}. First, we need some preliminaries on the Cauchy transform and Vitushkin's localization operator.

Let $\mu$ be a complex Borel measure on $\mathbb{C}$ with compact support. The \textit{Cauchy transform} of $\mu$ is defined by
$$\mathcal{C}\mu(z):=\int \frac{1}{\zeta-z}d\mu(\zeta).$$
An elementary application of Fubini's theorem shows that the above integral converges for almost every $z \in \mathbb{C}$ with respect to Lebesgue measure. Furthermore, $\mathcal{C} \mu$ is holomorphic outside the support of $\mu$ and satisfies $\mathcal{C} \mu(\infty)=0$ and $\mathcal{C} \mu'(\infty):=\lim_{z \to \infty}z(\mathcal{C}\mu(z)-\mathcal{C} \mu(\infty))=-\mu(\mathbb{C})$.

The definition of the Cauchy transform also makes sense if $\mu$ is a compactly supported distribution. In this case, we define
$$\mathcal{C} \mu := -\frac{1}{z} * \mu.$$
The following elementary lemma is well-known, see e.g. \cite[Theorem 18.5.4]{CON}.

\begin{lemma}
\label{CauchyDist}
We have
$$\overline{\partial} \frac{1}{\pi z} = \delta_0$$
in the sense of distributions, where $\delta_0$ is the Dirac delta at the origin. As a consequence, if $\mu$ is a compactly supported distribution on $\mathbb{C}$, then
$$\overline{\partial} (\mathcal{C}\mu) = -\pi \mu.$$
Also, if $f \in L^1_{loc}(\mathbb{C})$ is holomorphic on a neighborhood of infinity and $f(\infty)=0$, then
$$\mathcal{C}(\overline{\partial} f) = -\pi f.$$
\end{lemma}
Now, given $f \in L^1_{loc}(\mathbb{C})$ and $\phi \in C_c^{\infty}(\mathbb{C})$, we define \textit{Vitushkin's localization operator} $V_\phi$ by
$$V_\phi f:= \phi f + \frac{1}{\pi}\mathcal{C}(f \overline{\partial}\phi).$$
The same definition holds more generally if $f$ is a distribution.

\begin{lemma}
\label{VitLocOp1}
Let $f \in L^1_{loc}(\mathbb{C})$ and $\phi \in C_c^{\infty}(\mathbb{C})$. Then
$$V_\phi f = -\frac{1}{\pi} \mathcal{C} (\phi \overline{\partial}f)$$
in the sense of distributions.
\end{lemma}

\begin{proof}
By Lemma \ref{CauchyDist}, we have
$$\overline{\partial} (V_\phi f) = f \overline{\partial} \phi + \phi \overline{\partial}f + \frac{1}{\pi} \overline{\partial} (\mathcal{C}(f \overline{\partial} \phi)) = \phi \overline{\partial} f = \overline{\partial} \left(-\frac{1}{\pi} \mathcal{C} (\phi \overline{\partial} f)\right).$$
But both $V_\phi f$ and $-\frac{1}{\pi} \mathcal{C} (\phi \overline{\partial} f)$ are holomorphic in a neighborhood of $\infty$ and vanish at that point, hence these two distributions must be equal, again by Lemma \ref{CauchyDist}.

\end{proof}




\begin{lemma}
\label{LemExtension}
Let $U \subset \mathbb{C}$ be open and let $E \subset \mathbb{C}$ be $H^{\infty}$-removable and compact. Then every bounded holomorphic function on $U \setminus E$ has an analytic extension to an element of $H^{\infty}(U)$.
\end{lemma}
Note that here it is not assumed that $E$ is contained in $U$ (compare with Proposition \ref{propextend}).

\begin{proof}
Assume without loss of generality that $U$ is bounded. Consider a grid of squares $\{Q_j\}$ covering the plane and of side length $l$. Let $\{\phi_j\} \subset C_c^{\infty}(\mathbb{C})$ be a partition of unity subordinated to $\{2Q_j\}$, i.e. $0 \leq \phi_j \leq 1$, $\operatorname{supp}(\phi_j) \subset 2Q_j$ for each $j$ and
$$\sum_j \phi_j \equiv 1$$
on $\mathbb{C}$. Define $f$ to be zero on $\mathbb{C} \setminus (U \setminus E)$. Then $V_{\phi_j} f$ is identically zero except for finitely many $j$'s and
$$f=-\frac{1}{\pi} \mathcal{C}(\overline{\partial} f)=-\frac{1}{\pi} \sum_j \mathcal{C}(\phi_j \overline{\partial}f) = \sum_j V_{\phi_j}f,$$
where we used Lemma \ref{CauchyDist} and Lemma \ref{VitLocOp1}. Also, for each $j$, we have
$$\operatorname{supp} (\overline{\partial}(V_{\phi_j}f)) = \operatorname{supp}(\overline{\partial}(\mathcal{C}(\phi_j \overline{\partial}f))) \subset \operatorname{supp} \phi_j \cap \operatorname{supp}\overline{\partial}f \subset 2Q_j \cap (E \cup \partial U).$$
Hence $V_{\phi_j} f$ is holomorphic outside $2Q_j \cap E$ whenever $j$ is such that $2Q_j \cap \partial U = \emptyset$. Now, note that for every $j$, $2Q_j \cap E$ is $H^{\infty}$-removable. Moreover, a simple estimate shows that $V_{\phi_j}f$ is bounded. It follows that for each $j$ such that $2Q_j \cap \partial U = \emptyset$, the function $V_{\phi_j}f$ must be identically zero, since it vanishes at infinity. Therefore,
$$f=\sum_{j: 2Q_j \cap \partial U \neq \emptyset} V_{\phi_j}f,$$
so that $f$ is holomorphic on $U$ except maybe in a $4l$-neighborhood of $\partial U$. Since $l$ is arbitrary, we obtain that $f$ is holomorphic on the whole open set $U$. Finally, it must be bounded since $E$ has empty interior.
\end{proof}
We can now proceed with the proof of Proposition \ref{union}.

\begin{proof}
Suppose that $E$ and $F$ are $H^{\infty}$-removable compact sets. Let $f$ be any bounded holomorphic function on $\mathbb{C} \setminus (E \cup F)=(\mathbb{C} \setminus F)\setminus E$. By Lemma \ref{LemExtension}, the function $f$ has a bounded analytic extension to $\mathbb{C} \setminus F$. Hence $f$ must be constant, by $H^{\infty}$-removability of $F$. This shows that $E \cup F$ is $H^{\infty}$-removable.

\end{proof}

\begin{remark}
A simple argument using Lemma \ref{LemExtension} and the Baire category theorem shows that any compact countable union of $H^{\infty}$-removable compact sets is $H^{\infty}$-removable.
\end{remark}

\begin{remark}
One can prove Proposition \ref{union} by exploiting the fact that $H^{\infty}$-removable compact sets are totally disconnected, as in the proof of Proposition \ref{unionS}. However, the main advantage of the above proof using Vitushkin's localization operator is that it is easily generalized to handle $A$-removable sets, as we will see in Sect. \ref{sec2}.
\end{remark}

\subsection{Relationship with Hausdorff measure}

In this subsection, we study $H^{\infty}$-removability from the point of view of Hausdorff measure and dimension. More precisely, we shall see that compact sets of dimension less than one are removable, whereas those of dimension bigger than one are not.

First, let us recall the definitions of Hausdorff measure and Hausdorff dimension in the plane. Let $F$ be a subset of the plane. For $s \geq 0$ and $0<\delta \leq \infty$, we define
$$\mathcal{H}_{\delta}^{s}(F):= \inf \left\{ \sum_j \operatorname{diam}(F_j)^s : F \subset \bigcup_j F_j, F_j \subset \mathbb{C}, \operatorname{diam}(F_j) \leq \delta \right\}.$$
The $s$-\textit{dimensional Hausdorff measure} of $F$ is
$$\mathcal{H}^s(F) := \sup_{\delta>0} \mathcal{H}_{\delta}^s(F) = \lim_{\delta \to 0} \mathcal{H}_{\delta}^s(F).$$
The \textit{Hausdorff dimension} of $F$ is the unique positive number $\operatorname{dim}_\mathcal{H}(F)$ such that

\begin{displaymath}
\mathcal{H}^{s}(F) = \left\{ \begin{array}{ll}
\infty & \textrm{if $s < \dim_{\mathcal{H}}(F) $}\\
0 & \textrm{if $s>\dim_{\mathcal{H}}(F)$}.\\
\end{array} \right.
\end{displaymath}
The following result is generally attributed to Painlev\'e.

\begin{theorem}[Painlev\'e]
\label{ThmPainleve}
If $\mathcal{H}^{1}(E)=0$, then $E$ is $H^{\infty}$-removable.
\end{theorem}

\begin{proof}
Let $f$ be any bounded holomorphic function on $\Omega$, say $|f| \leq M$. Let $\epsilon>0$. Since $\mathcal{H}^{1}(E)=0$, we can cover $E$ by open disks $D_1,\dots,D_n$ of radius $r_1,\dots, r_n$ respectively such that
$$\sum_{j=1}^n r_j < \epsilon.$$
Let $\Gamma$ be the outer boundary of the union of these disks. Then we have
$$|f'(\infty)|= \left| \frac{1}{2\pi i} \int_{\Gamma}f(z) dz \right| \leq M \sum_{j=1}^n r_j < M\epsilon,$$
where $f'(\infty):=\lim_{z \to \infty}z(f(z)-f(\infty))$. As $\epsilon>0$ is arbitrary, this shows that $f'(\infty)=0$, for every bounded holomorphic function $f$ on $\Omega$. It follows that $E$ is $H^{\infty}$-removable. Indeed, if $f$ is a nonconstant bounded holomorphic function on $\Omega$, then there is a point $z_0 \in \Omega$ with $f(z_0) \neq f(\infty)$, but then the function
$$z \mapsto \frac{f(z)-f(z_0)}{z-z_0}$$
is a bounded holomorphic function on $\Omega$ whose derivative at infinity is nonzero.

\end{proof}

\begin{corollary}
\label{corPainleve}
If $\dim_{\mathcal{H}}(E)<1$, then $E$ is $H^{\infty}$-removable.
\end{corollary}
It is natural to ask whether the converse of Theorem \ref{ThmPainleve} holds, as this would yield an elegant solution to Painlev\'e's problem. Unfortunately, there are examples of $H^{\infty}$-removable compact sets with positive one-dimensional Hausdorff measure. The first one was given by Vitushkin \cite{VIT} in 1959. It was later observed by Garnett \cite{GAR2} that the planar Cantor quarter set is another example. See also \cite[Chapter IV]{GAR} for more information.

On the other hand, linear compact sets are $H^{\infty}$-removable precisely when they have zero length.

\begin{proposition}
Suppose that $E$ is contained in the real line. Then $E$ is $H^{\infty}$-removable if and only if its one-dimensional Lebesgue measure is zero.
\end{proposition}

\begin{proof}
If the measure of $E$ is zero, then $E$ must be $H^{\infty}$-removable by Theorem \ref{ThmPainleve}. Conversely, if the measure of $E$ is positive, consider the function
$$f(z):=\int_{E} \frac{dt}{t-z} \qquad (z \in \Omega),$$
which is just the Cauchy transform of the characteristic function of $E$ times the Lebesgue measure on $\mathbb{R}$.
Then $f$ is holomorphic on $\Omega$, $f(\infty)=0$ and $f'(\infty)<0$, hence $f$ is nonconstant. Moreover, a simple calculation shows that the imaginary part of $f$ is bounded on $\Omega$, thus $g:=e^{if}$ is a nonconstant bounded holomorphic function on $\Omega$. It follows that $E$ is not $H^{\infty}$-removable.
\end{proof}
Corollary \ref{corPainleve} shows that compact sets of sufficiently small dimension are removable. On the other hand, the following theorem states that compact sets of large enough dimension cannot be removable.
\begin{theorem}
\label{ThmFrostman}
If $\dim_{\mathcal{H}}(E)>1$, then $E$ is not $H^{\infty}$-removable.
\end{theorem}

\begin{proof}
By Frostman's lemma (see e.g. \cite[Theorem 1.23]{TOL}), there exists a nontrivial positive Borel measure $\mu$ supported on $E$ with growth $\mu(\mathbb{D}(z_0,r)) \leq r^s$ for all $z_0 \in \mathbb{C}, r>0$, for some $1<s<2$. Consider the function
$$f(z):= \mathcal{C}\mu(z)= \int \frac{d\mu}{\zeta-z} \qquad (z \in \Omega).$$
Then $f$ is a nonconstant holomorphic function on $\Omega$. Let us prove that the growth property of $\mu$ implies that $f$ is H\"{o}lder continuous on $\Omega$. Fix $z,w \in \Omega$ and write $\delta:=|z-w|$. Then
$$|f(z)-f(w)| \leq \delta \int \frac{d\mu(\zeta)}{|\zeta-z||\zeta-w|}.$$
We separate the integral over the four disjoint sets :

\begin{eqnarray*}
A_1 &:=& \{\zeta \in E : |\zeta-z|<\delta/2\}\\
A_2 &:=& \{\zeta \in E : |\zeta-w|<\delta/2\}\\
A_3 &:=& \{\zeta \in E: |\zeta-z| \leq |\zeta-w|, |\zeta-z| \geq \delta/2\}\\
A_4 &:=& \{\zeta \in E: |\zeta-z| > |\zeta-w|, |\zeta-w| \geq \delta/2\}.\\
\end{eqnarray*}
On $A_1$, $|\zeta-w|> \delta/2$, so we have
\begin{eqnarray*}
\delta \int_{A_1} \frac{d\mu(\zeta)}{|\zeta-z||\zeta-w|} &\leq& 2 \int_{A_1} \int_{|\zeta-z|}^\infty t^{-2}dt \, d\mu(\zeta)\\
&=& 2\int_{A_1} \int_{|\zeta-z|}^{\delta/2} t^{-2}dt \, d\mu(\zeta) + 2\int_{A_1} \int_{\delta/2}^{\infty} t^{-2}dt \, d\mu(\zeta)\\
&\leq& 2 \int_{0}^{\delta/2} \mu(\mathbb{D}(z,t))t^{-2} dt + 4\delta^{-1} \mu(A_1)\\
&\leq& 2 \int_{0}^{\delta/2} t^{s-2} dt + 4\delta^{-1}\delta^{s}2^{-s}\\
&=& C\delta^{s-1},\\
\end{eqnarray*}
where $C$ is independent of $\delta$.
Similarly for the integral over $A_2$. For the integral over $A_3$, we have
\begin{eqnarray*}
\delta \int_{A_3} \frac{d\mu(\zeta)}{|\zeta-z||\zeta-w|} &\leq& \delta \int_{\{|\zeta-z|\geq \delta/2\}} \frac{d\mu(\zeta)}{|\zeta-z|^2}\\
&=& 2\delta\int_{\delta/2}^\infty \mu(\{\delta/2 \leq |\zeta-z|<t\}) t^{-3} dt\\
&\leq& 2\delta\int_{\delta/2}^\infty t^{s-3} dt\\
&=& C' \delta^{s-1}.
\end{eqnarray*}
Here we used the fact that $s<2$. Similarly for the integral over $A_4$.

Combining the integrals together shows that $f$ is H\"{o}lder continuous on $\Omega$ and therefore extends continuously to a bounded function on $\mathbb{C}_\infty$. In particular, $f$ is a nonconstant bounded holomorphic function on $\Omega$ and $E$ is not $H^{\infty}$-removable.
\end{proof}

\begin{remark}
It is easier to prove directly that $f$ is bounded, as in \cite[Theorem 1.25]{TOL}. However, the fact that $f$ extends continuously to the whole Riemann sphere shows that $E$ is not even removable for the class $A$.
\end{remark}

\subsection{The case of dimension one}

We proved in the preceding subsection that compact sets of dimension less than one are $H^{\infty}$-removable whereas sets of dimension bigger than one are not. Painlev\'e's problem is thus reduced to the case of dimension exactly equal to one. In this case, however, the situation is much more complicated and we shall therefore content ourselves with a brief overview of the main results, without giving any proof. For the complete story, we refer the reader to \cite{TOL}.

One of the major advances toward a geometric characterization of $H^{\infty}$-removable sets of dimension one was the proof of the so-called \textit{Vitushkin's conjecture}.

\begin{theorem}[Vitushkin's conjecture]
\label{VitConj}
Assume that $\mathcal{H}^1(E)<\infty$. Then $E$ is $H^{\infty}$-removable if and only if $\mathcal{H}^1(E \cap \Gamma)=0$ for all rectifiable curves $\Gamma$.
\end{theorem}

The forward implication was previously known as Denjoy's conjecture and follows from the results of Calder\'on \cite{CAL} on the $L^2$-boundedness of the Cauchy transform operator. The other implication was proved by David \cite{DAV} in 1998.

In 2000, Joyce and M\"{o}rters \cite{JOM} constructed a nonremovable compact set $E$ which intersects every rectifiable curve in a set of zero one-dimensional Hausdorff measure, therefore showing that Vitushkin's conjecture is false witout the assumption of finite length.

In general, for sets of infinite length, there is no nice characterization of removability such as Theorem \ref{VitConj}. However, in the remarkable paper \cite{TOL2}, Tolsa obtained a metric characterization of $H^{\infty}$-removability in terms of the notion of curvature of a measure. Before stating the result, we need some definitions.

\begin{definition}
A positive Borel measure $\mu$ on $\mathbb{C}$ has \textit{linear growth} if there exists some constant $C$ such that $\mu(\mathbb{D}(z,r)) \leq Cr$ for all $z \in \mathbb{C}$ and all $r >0$.
\end{definition}
\begin{definition}
For a positive Radon measure $\mu$ on $\mathbb{C}$, we define the \textit{curvature of} $\mu$ by
$$c^2(\mu):=\int \int \int \frac{1}{R(x,y,z)^2} d\mu(x) d\mu(y) d\mu(z),$$
where $R(x,y,z)$ is the radius of the circle passing through $x,y,z$.
\end{definition}
We can now state Tolsa's result.
\begin{theorem}[Tolsa \cite{TOL2}]
\label{ThmTolsa}
A compact set $E \subset \mathbb{C}$ is not $H^{\infty}$-removable if and only if it supports a nontrivial positive Radon measure with linear growth and finite curvature.
\end{theorem}
A fundamental ingredient of the proof of Theorem \ref{ThmTolsa} is the notion of analytic capacity, which we now define.

\begin{definition}
The \textit{analytic capacity} of a compact set $E \subset \mathbb{C}$ is defined by
$$\gamma(E):=\sup\{|f'(\infty)| : f \in H^{\infty}(\Omega),|f| \leq 1\},$$
where $f'(\infty)=\lim_{z \to \infty} z(f(z)-f(\infty))$ and $\Omega=\mathbb{C}_\infty \setminus E$.
\end{definition}

Analytic capacity was first introduced by Ahlfors \cite{AHL} in 1947 for the study of Painlev\'e's problem, based on the observation that $E$ is $H^{\infty}$-removable if and only if $\gamma(E)=0$.

The characterization in Theorem \ref{ThmTolsa} is a consequence of the comparability between $\gamma$ and another capacity $\gamma_{+}$ which can be described in terms of measures with finite curvature and linear growth. Another important consequence of this comparability is the semiadditivity of analytic capacity, which was conjectured by Vitushkin \cite{VIT2} in 1967.

\begin{theorem}[Tolsa \cite{TOL2}]
\label{Semi}
We have
$$\gamma(E \cup F) \leq C(\gamma(E)+\gamma(F))$$
for all compact sets $E,F$, where $C$ is some universal constant.
\end{theorem}
Tolsa actually proved that analytic capacity is countably semiadditive.

We also remark that it is not known whether one can take $C=1$ in Theorem \ref{Semi}, i.e. whether analytic capacity is subadditive. See \cite{YOU} for more information on this problem.

Finally, we end this section by mentioning another important result of Tolsa regarding analytic capacity.

\begin{theorem}[Tolsa \cite{TOL3}]
\label{Bi}
Let $\phi:\mathbb{C} \to \mathbb{C}$ be a bilipschitz map. Then there exists a positive constant $C$ depending only on the bilipschitz constant of $\phi$ such that
$$C^{-1} \gamma(E) \leq \gamma(\phi(E)) \leq C \gamma(E)$$
for all compact sets $E \subset \mathbb{C}$.
In particular, the property of being $H^{\infty}$-removable is bilipschitz invariant.
\end{theorem}

\section{$A$-removable sets}
\label{sec2}
This section is dedicated to the study of $A$-removable compact sets. Recall that a compact set $E$ is $A$-removable if every continuous function on the sphere that is holomorphic outside $E$ is constant. It seems that $A$-removable compact sets were first studied by Besicovitch \cite{BES}, who proved that compact sets of $\sigma$-finite length are $A$-removable. The interest in $A$-removability was reinvigorated by the work of Vitushkin \cite{VIT2} several years later, motivated by applications to the theory of uniform rational approximation of holomorphic functions.

\subsection{Main properties}

It follows from the inclusion $A(\Omega) \subset H^{\infty}(\Omega)$ that a compact set $E$ is $A$-removable whenever it is $H^{\infty}$-removable. However, the converse is easily seen to be false. For instance, any segment (or more generally, any analytic arc) is $A$-removable by Morera's theorem but is never $H^{\infty}$-removable. In particular, $A$-removable sets need not be totally disconnected.

On the other hand, we will see in this subsection that $A$-removable sets and $H^{\infty}$-removable sets share many interesting properties. A first example is the fact that removability for the class $A$ is also a local property.

\begin{proposition}
The following are equivalent :

\begin{enumerate}[\rm(i)]
\item For any open set $U$ with $E \subset U$, every continuous function on $U$ which is holomorphic on $U \setminus E$ is actually analytic on the whole open set $U$;
\item $E$ is $A$-removable.
\end{enumerate}
\end{proposition}

\begin{proof}
The proof is exactly the same as in Proposition \ref{propextend}.
\end{proof}
We also have the following analogue of Proposition \ref{union}.
\begin{proposition}
\label{unionA}
If $E$ and $F$ are $A$-removable compact sets, then $E \cup F$ is also $A$-removable.
\end{proposition}

The above is a direct consequence of the following analogue of Lemma \ref{LemExtension}.

\begin{lemma}
\label{LemmeExtension}
Let $U \subset \mathbb{C}$ be open and let $E \subset \mathbb{C}$ be compact and $A$-removable. Then every continuous function on $\mathbb{C}_\infty$ that is holomorphic on $U \setminus E$ is actually analytic the whole open set $U$.
\end{lemma}

\begin{proof}
The argument is the same as in Lemma \ref{LemExtension}, except we need the fact that the functions $V_{\phi_j}f$ are continuous on some open set containing $U \cup E$. But this follows directly from the definition of Vitushkin's localization operator
$$V_{\phi_j} f= \phi_j f + \frac{1}{\pi}\mathcal{C}(f \overline{\partial}\phi_j)$$
and from the fact that the integral
$$\int \frac{f(\zeta)\overline{\partial}\phi_j(\zeta)}{\zeta-z} dm(\zeta)$$
depends continuously on $z$, since $f\overline{\partial}\phi_j$ is bounded. Here $m$ is the two-dimensional Lebesgue measure.

\end{proof}

\begin{remark}
A simple argument using Lemma \ref{LemmeExtension} and the Baire category theorem shows that any compact countable union of $A$-removable compact sets is $A$-removable.
\end{remark}

\subsection{Relationship with Hausdorff measure}

As for $H^{\infty}$-removability, there is a close relationship between Hausdorff measure and dimension and $A$-removability. Indeed, first note that Theorem \ref{ThmPainleve} implies that compact sets of zero one-dimensional Hausdorff measure are $A$-removable. In particular, any compact set of Hausdorff dimension strictly less than one is removable. On the other hand, the proof of Theorem \ref{ThmFrostman} shows that compact sets of Hausdorff dimension strictly bigger than one are never $A$-removable. As for sets of dimension exactly equal to one, we have the following sufficient condition.

\begin{theorem}[Besicovitch \cite{BES}]
\label{RemA}
If $\mathcal{H}^1(E)<\infty$, then $E$ is $A$-removable.
\end{theorem}

\begin{proof}
We only give a sketch of the proof. The interested reader may consult \cite[Chapter III, Section II]{GAR} for all the details.

Let $f:\mathbb{C}_\infty \to \mathbb{C}$ be continuous and holomorphic outside $E$, let $\omega_f:(0,\infty) \to [0,\infty)$ denote the modulus of continuity of $f$
$$\omega_f(\delta):=\sup \{|f(z)-f(w)|: |z-w|<\delta\}$$
and let $h(\delta):=\delta\omega_f(\delta)$. Fix $\epsilon>0$ and let $S$ be any square in $\mathbb{C}$. Since $h(\delta)/\delta \to 0$ as $\delta \to 0$ and $\mathcal{H}^1(E)<\infty$, we can find a cover $(S_j)_{j=1}^{\infty}$ of $S \cap E$ by squares of side length $\delta_j$ centered at $a_j$ with disjoint interiors and contained in $S$ such that
$$\sum_j h(\delta_j) < \epsilon.$$
Then by Cauchy's theorem,
$$\left|\int_{\partial S} f(\zeta)d\zeta\right|=\left|\sum_{j} \int_{\partial S_j} f(\zeta)d\zeta\right|=\left|\sum_{j} \int_{\partial S_j}(f(\zeta)-f(a_j))d\zeta\right| \leq \sum_{j} 4h(\delta_j) < 4\epsilon.$$
Since $\epsilon>0$ and $S$ are arbitrary, it follows from Morera's theorem that $f$ is holomorphic everywhere and hence must be constant by Liouville's theorem.

This shows that $E$ is $A$-removable.

\end{proof}
By the remark following the proof of Lemma \ref{LemmeExtension}, we get
\begin{corollary}
\label{sigmaf}
Suppose that $E$ is a countable union of compact sets of finite one-dimensional Hausdorff measure. Then $E$ is $A$-removable.
\end{corollary}

We end this subsection by mentioning a characterization of removability for product sets. If $E_1$ and $E_2$ are two compact subsets of $\mathbb{R}$ and $E_1$ is countable, then $E:=E_1 \times E_2$ is $A$-removable by Corollary \ref{sigmaf}. The converse is true provided that $m(E_2)>0$.

\begin{theorem}[Carleson \cite{CAR}]
\label{thmCarleson}
Let $E_1,E_2 \subset \mathbb{R}$ be compact and suppose that $m(E_2)>0$. Then $E=E_1 \times E_2$ is $A$-removable if and only if $E_1$ is countable.
\end{theorem}

\subsection{Continuous analytic capacity}
\label{subsecCA}

This subsection is a brief introduction to an extremal problem whose solution measures the size of a compact set from the point of view of $A$-removability. This extremal problem is usually referred to as \textit{continuous analytic capacity} and is the analogue of analytic capacity for the class $A$.

\begin{definition}
The \textit{continuous analytic capacity} of a compact set $E \subset \mathbb{C}$ is defined by
$$\alpha(E):=\sup \{|f'(\infty)| : f \in A(\Omega), |f| \leq 1\}.$$
\end{definition}
Continuous analytic capacity was first introduced by Erokhin and Vitushkin \cite{VIT2} in order to study problems of uniform rational approximation of holomorphic functions. See also \cite{ZAL},\cite{DAVIE} and \cite[Chapter VIII]{GAM} for the applications of continuous analytic capacity to this type of problem.

It follows easily from the definition that $\alpha(E) \leq \gamma(E)$ and that $E$ is $A$-removable if and only if $\alpha(E)=0$. Unfortunately, unlike analytic capacity, there is no known geometric characterization of compact sets $E$ such that $\alpha(E)=0$. The main recent advances are again due to Tolsa, who proved that continuous analytic capacity is (countably) semiadditive and that it is bilipschitz invariant in the sense of Theorem \ref{Bi}. See \cite{TOL4} and \cite{TOL3}.

\section{$S$-removable sets}
\label{sec3}
This section is dedicated to the study of $S$-removable compact sets. Recall that a compact set $E \subset \mathbb{C}$ is $S$-removable if every conformal map on $\Omega=\mathbb{C}_\infty \setminus E$ is a M\"{o}bius transformation.

The notion of $S$-removability was first considered by Sario \cite{SAR} for the problem of classifying Riemann surfaces. A couple of years later, in 1950, Alhfors and Beurling published their seminal paper \cite{AHB} containing among other things a characterization of $S$-removable compact sets, which we present below. First, we need some preliminaries on quasiconformal mappings.

\subsection{Preliminaries on quasiconformal mappings}

This subsection consists of a very brief introduction to quasiconformal mappings. For more information, the reader may consult \cite{AHL2}, \cite{AST} or \cite{LEH}.

There are several equivalent definitions of quasiconformal mappings. The following is the analytic one.
\begin{definition}
Let $K \geq 1$, let $U,V$ be domains in the Riemann sphere and let $f:U \to V$ be an orientation-preserving homeomorphism. We say that $f$ is $K$-\textit{quasiconformal} on $U$ if it belongs to the Sobolev space $W^{1,2}_{loc}(U)$ and satisfies the Beltrami equation
$$\overline{\partial}f=\mu \, \partial f$$
almost everywhere on $U$ for some measurable function $\mu$ with $\|\mu\|_{\infty} \leq \frac{K-1}{K+1}$. In this case, the function $\mu$ is called the \textit{Beltrami coefficient} of $f$ and is denoted by $\mu_f$.
\end{definition}
A mapping is conformal if and only if it is $1$-quasiconformal. This is usually referred to as Weyl's lemma. We will also need the fact that quasiconformal mappings preserve sets of area zero, in the sense that if $E \subset U$ is measurable, then $m(E)=0$ if and only if $m(f(E))=0$, where $m$ is the two-dimensional Lebesgue measure.

The following fundamental theorem was first proved by Morrey \cite{MOR} in 1938.

\begin{theorem}[Measurable Riemann mapping theorem]
\label{MRMT}
Let $U$ be a domain in the sphere and let $\mu : U \to \mathbb{C}$ be a measurable function with $\|\mu\|_{\infty}<1$. Then there exists a quasiconformal mapping $f$ on $U$ such that $\mu=\mu_f$, i.e.
$$\overline{\partial}f = \mu \, \partial f$$
almost everywhere on $U$. Moreover, a quasiconformal mapping $g$ on $U$ satisfies $\mu_g=\mu=\mu_f$ if and only if $f \circ g^{-1}: g(U) \to f(U)$ is conformal.
\end{theorem}
We can now describe the various properties of $S$-removable sets.

\subsection{Main properties}

As in Proposition \ref{TotDisc}, any $S$-removable set must be totally disconnected. Indeed, if $F$ is a nontrivial connected component of an $S$-removable compact set $E$, then by the Riemann mapping theorem there is a conformal map $f$ on $\mathbb{C}_\infty \setminus F \supset \mathbb{C}_\infty \setminus E$ and the image of this map can always be chosen so that $f$ is not conformal everywhere.

Although $S$-removable sets are totally disconnected, it is interesting to mention the result of Thurston \cite{THU} saying that there exist many connected sets $E$ (even \textit{quasi-intervals}) that are not far from being removable in the sense that there is an $\epsilon>0$ such that every conformal map on $\Omega$ with Schwarzian derivative less than $\epsilon$ is a M\"{o}bius transformation. Such compact sets are said to have \textit{conformally rigid} complement and they are known to have zero area \cite{OVE}.

\begin{proposition}
\label{RemS}
If $E$ is $H^{\infty}$-removable, then $E$ is $S$-removable.
\end{proposition}

\begin{proof}
Let $f$ be any conformal map on $\Omega$. Composing $f$ with a M\"{o}bius transformation if necessary, we can assume that $f(\infty)=\infty$. Then $f$ is bounded near $E$. Fix some point $z_0 \in \Omega$ and consider the function
$$z \mapsto \frac{f(z)-f(z_0)}{z-z_0}.$$
Clearly, this defines a bounded holomorphic function on $\Omega$, which must be constant by $H^{\infty}$-removability of $E$. It follows that $f$ is linear. This shows that $E$ is $S$-removable.

\end{proof}

In particular, by Theorem \ref{ThmPainleve}, any compact set of zero one-dimensional Hausdorff measure is $S$-removable. On the other hand, the following proposition implies that $S$-removable sets must be rather small.

\begin{proposition}
\label{RemArea}
If $E$ is $S$-removable, then the area of $E$ is zero.
\end{proposition}

\begin{proof}
If $E$ has positive area, then by the measurable Riemann mapping theorem (Theorem \ref{MRMT}) there exists a quasiconformal mapping $f$ on $\mathbb{C}_\infty$ such that $\mu_f=\frac{1}{2} \chi_E$ almost everywhere. In particular, the map $f$ is conformal outside $E$ and is not a M\"{o}bius transformation, since $\overline{\partial} f \neq 0$ on a set of positive measure. Hence $E$ cannot be $S$-removable.
\end{proof}

\begin{remark}

The above proof actually shows that sets of positive area are never $CH$-removable.
\end{remark}

\begin{remark}
An alternative argument is the following. By a result of \cite{UY}, a compact set has zero area if and only if it is removable for Lipschitz functions on the sphere which are holomorphic outside the set. Therefore, if $E$ has positive area, then there exists such a function, say $f$, which is not holomorphic everywhere. It follows that if $\epsilon>0$ is small enough, then $z \mapsto z+\epsilon f(z)$ is a non-M\"{o}bius homeomorphism of the sphere conformal outside $E$.
\end{remark}

\begin{remark}
An interesting question is the following : if $E$ has positive area, does there necessarily exist a homeomorphism of the sphere onto itself which is conformal outside $E$ but is not quasiconformal everywhere? The answer is yes, by a result of Kaufman and Wu \cite{KAU} stating that there always exists a function $f \in CH(\Omega)$ which maps a subset $F$ of $E$ of positive area onto a set of zero area. This map $f$ cannot be quasiconformal everywhere since it doesn't preserve sets of measure zero. It is still open whether one can take $F=E$ in Kaufman and Wu's result, see \cite{BIS2}.
\end{remark}

The proof of Proposition \ref{RemArea} illustrates the usefulness of quasiconformal mappings in the study of $S$-removable sets. This motivates the following definition.

\begin{definition}
We say that a compact set $E \subset \mathbb{C}$ is $QC$-\textit{removable} if every quasiconformal mapping on $\mathbb{C}_\infty \setminus E$ has a quasiconformal extension to the whole Riemann sphere.
\end{definition}

The following result essentially says that the property of being $QC$-removable is local.

\begin{proposition}
\label{PropExtendQC}
The following are equivalent :
\begin{enumerate}[\rm(i)]
\item For any open set $U$ with $E \subset U$, every quasiconformal mapping on $U \setminus E$ extends quasiconformally to the whole open set $U$;
\item $E$ is $QC$-removable.
\end{enumerate}
\end{proposition}

\begin{proof}
The direct implication is trivial, while the converse follows from classical quasiconformal extension theorems, see e.g. \cite[Chapter II, Theorem 8.3]{LEH}.
\end{proof}
A remarkable consequence of the measurable Riemann mapping theorem is that the notions of $QC$-removability and $S$-removability actually coincide.

\begin{proposition}
\label{propQC}
A compact set $E$ is $QC$-removable if and only if it is $S$-removable.
\end{proposition}

\begin{proof}
Assume that $E$ is $QC$-removable. First note that the area of $E$ must be zero, by a classical result of Koebe saying that every domain can be mapped conformally onto the complement of a set of zero area. Now, if $f$ is any conformal map on $\Omega$, then in particular $f$ is quasiconformal outside $E$ so that it extends quasiconformally to the whole sphere, by $QC$-removability of $E$. But then by Weyl's lemma, the map $f$ must be a M\"{o}bius transformation. Thus $E$ is $S$-removable.

Conversely, if $E$ is $S$-removable, let $g$ be any quasiconformal mapping on $\Omega$. By Theorem \ref{MRMT}, there exists a quasiconformal mapping $f:\mathbb{C}_\infty \to \mathbb{C}_\infty$ such that $f \circ g$ is conformal on $\Omega$. Since $E$ is $S$-removable, the map $f \circ g$ is a M\"{o}bius transformation and thus $g= f^{-1} \circ (f \circ g)$ extends quasiconformally to the whole sphere. Since $g$ was arbitrary, we get that $E$ must be $QC$-removable.

\end{proof}

An interesting consequence of the above result is that Proposition \ref{PropExtendQC} remains true without the assumption that $E \subset U$ in $(i)$. Before proving this, we need the following topological lemma.

\begin{lemma}
\label{lemtotdisc}
Let $X$ be a totally disconnected compact Hausdorff space. Suppose that $F_1$ and $F_2$ are two disjoint closed subsets of $X$. Then there exist disjoint closed subsets $X_1$ and $X_2$ of $X$ such that $X = X_1 \cup X_2$, $F_1 \subset X_1$ and $F_2 \subset X_2$.
\end{lemma}

\begin{proof}
The result easily follows from the fact that such spaces $X$ are zero-dimensional.

\end{proof}

We can now prove the following generalization of Proposition \ref{PropExtendQC}.

\begin{proposition}
\label{PropExtendQC2}
The following are equivalent :
\begin{enumerate}[\rm(i)]
\item For any open set $U \subset \mathbb{C}$, every quasiconformal mapping on $U \setminus E$ extends quasiconformally to the whole open set $U$;
\item $E$ is $S$-removable.
\end{enumerate}
\end{proposition}

\begin{proof}
The direct implication follows trivially from Proposition \ref{propQC}. Conversely, assume that $E$ is $S$-removable. Let $U$ be any open set, let $f$ be a quasiconformal mapping on $U \setminus E$ and let $\epsilon>0$. Define $F_1:=\{z \in E : \operatorname{dist}(z,\mathbb{C} \setminus U) \geq \epsilon\}$ and $F_2:=E \setminus U$. Then $F_1$ and $F_2$ are two disjoint closed subsets of $E$, so by Lemma \ref{lemtotdisc} there exist two disjoint closed subsets of $E$, say $E_1$ and $E_2$, such that $E = E_1 \cup E_2$, $F_1 \subset E_1$ and $F_2 \subset E_2$. Here we used the fact that $S$-removable sets are totally disconnected. Clearly, $E_1$ is $S$-removable. Since $f$ is quasiconformal on $(U \setminus E_2) \setminus E_1$ and $E_1$ is a compact subset of the open set $U \setminus E_2$, it follows from Proposition \ref{propQC} and Proposition \ref{PropExtendQC} that $f$ extends quasiconformally to $U \setminus E_2$. This extension is well-defined for every $z \in U$ such that $\operatorname{dist}(z,\mathbb{C} \setminus U) > \epsilon$, independently of the precise value of $\epsilon$ since $E$ has empty interior. Since $\epsilon>0$ is arbitrary, it follows that $f$ extends quasiconformally to the whole open set $U$.

\end{proof}
As in Proposition \ref{union} and Proposition \ref{unionA}, we can deduce from this that unions of removable sets are removable.

\begin{proposition}
\label{unionS}
If $E$ and $F$ are $S$-removable compact sets, then $E \cup F$ is also $S$-removable.
\end{proposition}

\begin{proof}
Any conformal map on $\mathbb{C}_\infty \setminus (E \cup F) = (\mathbb{C}_\infty \setminus E) \setminus F$ extends quasiconformally to $\mathbb{C}_\infty \setminus E$ and then to the whole sphere, by Proposition \ref{PropExtendQC2}. Since $E \cup F$ has zero area, the extension must be conformal everywhere.

\end{proof}

\begin{remark}
Again, a simple argument using Proposition \ref{PropExtendQC2} and the Baire category theorem shows that any compact countable union of $S$-removable compact sets is $S$-removable.
\end{remark}

\subsection{A characterization of removability}

In this subsection, we present a characterization of $S$-removable sets due to Ahlfors and Beurling \cite{AHB} and based on ideas of Grunsky. First, we need a definition.

\begin{definition}
We say that a compact set $E$ has \textit{absolute area zero} if for every conformal map $f \in S(\Omega)$, the complement of $f(\Omega)$ has measure zero.
\end{definition}
Note that sets of absolute area zero must be totally disconnected. We also mention that a sufficient condition for a set $E$ to have absolute area zero in terms of the moduli of nested annuli surrounding each point of $E$ can be found in McMullen's book \cite[Theorem 2.16]{MCM}.

\begin{theorem}[Ahlfors--Beurling \cite{AHB}]
\label{AhlforsB}
A compact set $E \subset \mathbb{C}$ is $S$-removable if and only if it has absolute area zero.
\end{theorem}

Note that the direct implication follows from Proposition \ref{RemArea}. Indeed, if $E$ is $S$-removable, then it must have zero area and every conformal map $f$ on $\Omega$ is a M\"{o}bius transformation, so that $\mathbb{C}_\infty \setminus f(\Omega) = f(E)$ has zero area.

For the converse, we follow \cite{AHB} and introduce two extremal problems.

Let $E \subset \mathbb{C}$ be compact and as usual denote by $\Omega$ the complement of $E$ in the Riemann sphere. We assume that $\Omega$ is connected. Define
$$\eta(E):=\sup |f'(\infty)|,$$
where the supremum is taken over all holomorphic functions $f$ on $\Omega$ with $f(\infty)=0$ and whose Dirichlet integral $D(f)$ satisfies
$$D(f):=\int_{\Omega} |f'|^2 dm \leq \pi.$$
Also, let
$$\beta(E):=\sup |f'(\infty)|$$
where the supremum is taken over all conformal maps $f \in S(\Omega)$ with $f(\infty)=0$ and having the property that the complement of $(1/f)(\Omega)$ has area greater or equal to $\pi$. If there is no such function, we set $\beta(E)=0$.

Ahlfors and Beurling's remarkable result states that $\eta(E)$ and $\beta(E)$ are actually equal, for any compact set $E$. This implies the reverse implication in Theorem \ref{AhlforsB}. Indeed, suppose that $E$ has absolute area zero. Then $\Omega$ is connected and $0=\beta(E)=\eta(E)$. Let $f$ be conformal on $\Omega$. Without loss of generality, assume that $f(\infty)=\infty$. Then $f$ is bounded near $E$, so its Dirichlet integral there is finite. Then, since $\eta(E)=0$, a simple modification of the proof of Proposition \ref{propextend} shows that $f$ has an analytic extension to the whole complex plane. This extension must be conformal, so that $f$ is linear. Therefore $E$ is $S$-removable.

\begin{theorem}[Ahlfors--Beurling \cite{AHB}]
For any compact set $E$, we have
$$\eta(E)=\beta(E).$$

\end{theorem}

\begin{proof}
We only give a sketch of the proof.

By a simple normal family argument, it suffices to prove the result for compact sets $E$ that are bounded by finitely many disjoint analytic Jordan curves. In this case, there exist conformal maps $g$ and $h$ of $\Omega$ onto domains bounded by horizontal slits and vertical slits respectively with normalization
$$g(z)=z+\frac{a}{z}+\frac{a_2}{z^2}+\dots$$
and
$$h(z)=z+\frac{b}{z}+\frac{b_2}{z^2}+\dots$$
near infinity. In this case, the maps $g$ and $h$ are unique, see e.g. \cite[Chapter 5, Section 2]{GOL}.

Let $\Gamma$ denote the boundary of $\Omega$ oriented positively and let $f$ be any function regular on $\overline{\Omega}$ and holomorphic on $\Omega$ with $f(\infty)=0$ and $D(f) \leq \pi$. A simple calculation using Green's theorem and the Cauchy-Riemann equations shows that
$$\iint_{\Omega}f'(z)(\overline{g'(z)-h'(z)})dx \, dy = \frac{i}{2} \int_{\Gamma}f(d\overline{g}-d\overline{h}).$$
Since $g(\Omega)$ is a horizontal slit domain and $h(\Omega)$ is a vertical slit domain, we have $d\overline{g}=dg$ and $d\overline{h}=-dh$ on $\Gamma$, thus we obtain
\begin{equation}
\label{eqDir}
\iint_{\Omega}f'(z)(\overline{g'(z)-h'(z)})dx \, dy = \frac{i}{2} \int_{\Gamma}f(dg+dh)=2\pi f'(\infty),
\end{equation}
where we used the fact that
$$\int_{\Gamma}f(z)(g'(z)+h'(z))dz = -2 \pi i (f(g'+h'))'(\infty)=-4\pi i f'(\infty).$$
Replacing $f$ by $g-h$ in (\ref{eqDir}) yields
$$D(g-h) = 2\pi(a-b).$$
Now, by (\ref{eqDir}) and the Cauchy-Schwarz inequality, we get
$$4\pi^2|f'(\infty)|^2 \leq 2\pi(a-b)D(f) \leq 2 \pi^2(a-b)$$
and thus
$$|f'(\infty)| \leq \sqrt{\frac{a-b}{2} },$$
with equality for the function
$$f=\frac{g-h}{\sqrt{2(a-b)}}.$$
One can prove using some approximation process that the above inequality holds even for functions $f$ that are only assumed to be holomorphic on $\Omega$. It follows that
$$\eta(E)=\sqrt{\frac{a-b}{2}}.$$
It only remains to show that the same equality holds with $\eta(E)$ replaced by $\beta(E)$. First, we need to introduce two integral quantities. For functions $\phi, \psi$ regular on $\overline{\Omega}$ and holomorphic on $\Omega$  except a simple pole at infinity, define
$$I(\phi,\psi):=\frac{i}{2}\int_{\Gamma} \phi\, d\overline{\psi}$$
and
$$I(\phi):=\frac{i}{2}\int_{\Gamma} \phi\, d\overline{\phi} = I(\phi,\phi).$$
Note that $I(\phi,\psi) = \overline{I(\psi,\phi)}$.

Now, let $f$ be regular on $\overline{\Omega}$ and holomorphic on $\Omega$ with a simple zero at infinity. Set $\phi:=1/f$. As in (\ref{eqDir}), we have
$$I(\phi,g+h) = \frac{i}{2} \int_{\Gamma} \phi(dg-dh) = -\pi c (a-b),$$
where $c$ is the residue of $\phi$ at infinity. In particular, this holds for $\phi=g+h$ and thus
$$I(g+h)=-2\pi(a-b).$$

Now, it is easy to see that if $\phi_0$ has a removable singularity at infinity, then $I(\phi_0)$ is just the Dirichlet integral of $\phi_0$. It follows that
$$I\left(\phi - \frac{c}{2}(g+h)\right) \geq 0.$$
Therefore,
\begin{eqnarray*}
I(\phi) &=& I\left(\phi-\frac{c}{2}(g+h)\right) + I\left(\phi,\frac{c}{2}(g+h)\right)+I\left(\frac{c}{2}(g+h),\phi\right)-I\left(\frac{c}{2}(g+h)\right)\\
&\geq& 0 - \frac{\pi}{2}|c|^2(a-b)+\frac{c}{2}\overline{I(\phi,g+h)} +\frac{\pi}{2}|c|^2(a-b)\\
&=& -\frac{\pi}{2} |c|^2(a-b).
\end{eqnarray*}

Suppose in addition that $f$ is conformal on $\Omega$ and that the area of the complement of $\phi(\Omega)$ is greater or equal to $\pi$. In this case, Green's theorem shows that $-I(\phi)$ is precisely the area enclosed by $\phi(\Gamma)$ and hence $I(\phi) \leq -\pi$. Combining this with the preceding inequality, we obtain

$$\frac{1}{|c|} \leq \sqrt{\frac{a-b}{2}}.$$
But a simple calculation yields $|c|=1/|f'(\infty)|$, so that
\begin{equation}
\label{eqAHB}
|f'(\infty)| \leq \sqrt{\frac{a-b}{2}}.
\end{equation}
It follows that,
$$\beta(E) \leq \sqrt{\frac{a-b}{2}}.$$
Finally, observe that equality in (\ref{eqAHB}) is attained by the function
$$f=\frac{\sqrt{2(a-b)}}{g+h}.$$
Using the argument principle, one can prove that this function is univalent on $\Omega$. Therefore, we obtain
$$\beta(E)=\sqrt{\frac{a-b}{2}} = \eta(E).$$

\end{proof}

\section{$CH$-removable sets}
\label{sec4}
The last section of this article deals with $CH$-removable sets. Recall that a compact set $E \subset \mathbb{C}$ is said to be $CH$-removable if every homeomorphism of the sphere onto itself that is conformal outside $E$ is a M\"{o}bius transformation.

Besides earlier results on removable product sets by Gehring \cite{GEH}, the notion of $CH$-removability (also sometimes referred to as \textit{conformal removability} or \textit{holomorphic removability}) seems to have first been seriously investigated by Kaufman \cite{KAU2}, who constructed several examples of nonremovable and removable sets. Bishop \cite{BIS2} later gave another construction of similar sets. In recent years, there has been a strong renewal of interest in $CH$-removability, mainly due to its applications in the theory of holomorphic dynamics. In that respect, we mention the work of Jones \cite{JON} whose results were later generalized by Jones and Smirnov \cite{JOS}. Furthermore, Jeremy Kahn recently introduced in his Ph.D. thesis new dynamical methods to prove that some Julia sets of complex quadratic polynomials are $CH$-removable, from which one can deduce that the boundary of the Mandelbrot set is locally connected at the corresponding parameters. This is quite reminiscent of Adrien Douady's philosophy that one ``first plows in the dynamical plane and then harvest in the parameter plane''.

\subsection{Main properties}

Clearly, a compact set $E$ is $CH$-removable whenever it is $S$-removable. In particular, compact sets of absolute area zero are always $CH$-removable in view of Theorem \ref{AhlforsB}. Moreover, it follows from Proposition \ref{RemS} that $H^{\infty}$-removable sets are $CH$-removable. In fact, we have the following stronger statement.

\begin{proposition}
\label{RemCH}
If $E$ is $A$-removable, then $E$ is $CH$-removable.
\end{proposition}

\begin{proof}
The proof is the same as in Proposition \ref{RemS}.
\end{proof}
Combining this with Corollary \ref{sigmaf}, we obtain
\begin{corollary}
\label{CorFiniteLengthCH}
If $E$ is a countable union of compact sets of finite one-dimensional Hausdorff measure, then $E$ is $CH$-removable.
\end{corollary}

A remarkable consequence of the measurable Riemann mapping theorem is that the property of being $CH$-removable is quasiconformally invariant.

\begin{proposition}
Let $h:\mathbb{C}_\infty \to \mathbb{C}_\infty$ be a quasiconformal mapping with $h(\infty)=\infty$. Then $E$ is $CH$-removable if and only if $h(E)$ is.
\end{proposition}

\begin{proof}
Since the inverse of a quasiconformal mapping is also quasiconformal, it suffices to prove one of the two implications, say the first one. Assume that $E$ is $CH$-removable. Note that $E$ and $h(E)$ must have zero area, by the remark following Proposition \ref{RemArea} and the fact that quasiconformal mappings preserve sets of measure zero.

Let $g:\mathbb{C}_\infty \to \mathbb{C}_\infty$ be any homeomorphism conformal outside $h(E)$. By Theorem \ref{MRMT}, there exists a quasiconformal mapping $f:\mathbb{C}_\infty \to \mathbb{C}_\infty$ such that $\mu_{h^{-1} \circ g^{-1}} = \mu_f$ outside $g(h(E))$. By the uniqueness part, $f\circ g \circ h$ is a homeomorphism of $\mathbb{C}_\infty$ onto itself which is conformal outside $E$. Since $E$ is $CH$-removable, $f \circ g \circ h$ must be conformal everywhere, so that $g$ is quasiconformal on $\mathbb{C}_\infty$. But $g$ is conformal outside $h(E)$, a set of zero area, hence it must be conformal everywhere by Weyl's lemma. This shows that $h(E)$ is $CH$-removable.

\end{proof}

\begin{corollary}
Quasicircles (images of the unit circle under quasiconformal mappings of the sphere) are $CH$-removable.
\end{corollary}
We also mention that \textit{David circles} are also $CH$-removable, see \cite{ZAK}. David circles are images of the unit circle under so-called \textit{David maps}, which are generalizations of quasiconformal mappings where the Beltrami coefficient is allowed to tend to one in a controlled way.

Now, recall that by Proposition \ref{propQC}, it suffices to assume in the definition of $S$-removability that the maps are quasiconformal outside the set. In that regard, it is natural to introduce the following definition.
\begin{definition}
We say that a compact set $E \subset \mathbb{C}$ is $QCH$-\textit{removable} if every homeomorphism of $\mathbb{C}_\infty$ onto itself that is quasiconformal outside $E$ is actually quasiconformal everywhere.
\end{definition}

As in Proposition \ref{PropExtendQC}, the property of being $QCH$-removable is local.

\begin{proposition}
\label{PropExtendQCH}
The following are equivalent :
\begin{enumerate}[\rm(i)]
\item For any open set $U$ with $E \subset U$, every homeomorphism $f:U \to f(U)$ that is quasiconformal on $U \setminus E$ is actually quasiconformal on the whole open set $U$;
\item $E$ is $QCH$-removable.
\end{enumerate}
\end{proposition}

We do not know however if Proposition \ref{PropExtendQCH} holds without the assumption that $E$ is contained in $U$ in (i), see the discussion following Question \ref{Q1}.

We also have the following analogue of Proposition \ref{propQC}.

\begin{proposition}
\label{propQCH}
A compact set $E$ is $QCH$-removable if and only if it is $CH$-removable.
\end{proposition}

\begin{proof}
The proof is exactly the same as in Proposition \ref{propQC}, except for the fact that $QCH$-removable sets have zero area, which follows from the third remark following Proposition \ref{RemArea}.
\end{proof}

\begin{remark}
Some authors claimed without proof that the direct implication (i.e. $QCH$-removability implies $CH$-removability) follows trivially from the definition. However, the difficult part is proving that $QCH$-removable sets have zero area. We do not know any elementary proof of this fact. Note that instead of resorting to Kaufman and Wu's theorem, one can also use known results on David maps.
\end{remark}

An important consequence of Proposition \ref{propQCH} is that the property of being $CH$-removable is invariant under disjoint unions.

\begin{corollary}
\label{QCHunion}
If $E$ and $F$ are disjoint $CH$-removable compact sets, then $E \cup F$ is $CH$-removable.
\end{corollary}

\begin{proof}
By Proposition \ref{propQCH}, it suffices to prove the result for $QCH$-removability. If $f:\mathbb{C}_\infty \to \mathbb{C}_\infty$ is a homeomorphism of the sphere onto itself that is quasiconformal on $\mathbb{C}_\infty \setminus (E \cup F) = (\mathbb{C}_\infty \setminus E) \setminus F$, then by $QCH$-removability of $F$ and Proposition \ref{PropExtendQCH}, the map $f$ is in fact quasiconformal on $\mathbb{C}_\infty \setminus E$. Since $E$ is also $QCH$-removable, we get that $f$ is actually quasiconformal everywhere.
\end{proof}

We conclude this subsection by a brief presentation of the removability theorems of Jones and Smirnov \cite{JOS}, which give elegant and geometric sufficient conditions for $CH$-removability.

First, we need some definitions. For the rest of this subsection, we suppose that $K$ is the boundary of a domain $\Omega$. We shall also assume for simplicity that $\Omega$ is simply connected, although the following also works for arbitrary domains, with suitable modifications.

\begin{definition}
A \textit{Whitney decomposition} of $\Omega$ consists of a countable collection of dyadic squares $\{Q_j\}$ contained in $\Omega$ such that

\begin{enumerate}[\rm(i)]
\item the interiors of the squares are pairwise disjoint;
\item the union of their closure is the whole domain $\Omega$;
\item for each $Q_j$, we have $\operatorname{diam}(Q_j) \simeq \operatorname{dist}(Q_j,\partial \Omega)$.
\end{enumerate}

\end{definition}

The existence of such a decomposition is well-known and is usually referred to as the \textit{Whitney covering lemma}.

Fix some point $z_0 \in \Omega$ and let $\Gamma:=\{\gamma_z : z \in K\}$ be the family of all hyperbolic geodesics $\gamma_z$ connecting $z_0$ to some point $z \in K$.

\begin{definition}
For each Whitney square $Q_j \subset \Omega$, the \textit{shadow} of $Q_j$ is defined by
$$S(Q_j):=\{z \in K : \gamma_z \cap Q_j \neq \emptyset\}.$$
\end{definition}

We can now state the main result of \cite{JOS}.

\begin{theorem}[Jones--Smirnov \cite{JOS}]
\label{jos}
Suppose that
$$\sum_{j} \operatorname{diam}(S(Q_j))^2 < \infty,$$
where the sum is over all Whitney squares $Q_j$ in $\Omega$. Then $K$ is $CH$-removable.
\end{theorem}
See also \cite{KOS} for a generalization of the above result.

An important consequence of Theorem \ref{jos} is that boundaries of sufficiently nice simply connected domains are $CH$-removable.

\begin{definition}
We say that a simply connected domain $\Omega$ is a \textit{H\"{o}lder domain} if the Riemann conformal map is H\"{o}lder continuous on the closed unit disk.
\end{definition}

\begin{corollary}[Jones--Smirnov \cite{JOS}]
\label{josHolder}
Boundaries of H\"{o}lder domains are $CH$-removable.
\end{corollary}

This supersedes an earlier result of Jones \cite{JON} saying that boundaries of \textit{John domains} are $CH$-removable.

\subsection{Nonremovable sets of zero area}

Recall that by the first remark following Proposition \ref{RemArea}, compact sets of positive area are not $CH$-removable. The converse is well-known to be false. In this subsection, we present some examples of nonremovable sets of zero area.

As far as we know, the first such examples were given by Carleson \cite{CAR} and Gehring \cite{GEH}, who proved the following characterization of $CH$-removable product sets.

\begin{theorem}[Carleson \cite{CAR}, Gehring \cite{GEH}]
\label{thmCarGeh}
If $F \subset \mathbb{R}$ is compact, then $E:=F \times [0,1]$ is $CH$-removable if and only if $F$ is countable.
\end{theorem}
By taking $F$ to be any uncountable compact set of zero one-dimensional Lebesgue measure, we obtain a nonremovable set $E$ of zero area.

The proof of the direct implication in Theorem \ref{thmCarGeh} involves the construction of a homeomorphism of the sphere onto itself which is quasiconformal outside $E$ of the form
$$x+iy \mapsto x+iy + g(y) \mu((-\infty,x)),$$
where $\mu$ is any continuous probability measure on $F$ and $g$ is some smooth function supported on $[0,1]$. If $F$ is uncountable, it is possible to choose $g$ such that $h$ is not quasiconformal everywhere.
The other implication is a direct consequence of Theorem \ref{thmCarleson} and Proposition \ref{RemCH}.

In fact, if $F$ is uncountable, then $F \times [0,1]$ contains a closed graph which is not $CH$-removable. This much stronger statement was proved by Kaufman \cite{KAU2}. See also \cite{WU} for other examples of nonremovable product sets.

As for examples of totally disconnected nonremovable sets of zero area, an example was given by Rothberger \cite{ROT} using only elementary normal family arguments. More precisely, the proof involves a simple and elegant geometric construction using a sequence of multiply connected slit domains converging to a Cantor set of zero area. A non-M\"{o}bius homeomorphism of the sphere conformal outside the set is then obtained as a suitable limit of slit mappings.

In the remaining of this subsection, we present a construction of nonremovable Jordan curves of zero area due to Bishop \cite{BIS2}. First, we need a definition.

\begin{definition}
A \textit{Hausdorff measure function} is an increasing continuous function $h: [0,\infty) \to [0,\infty)$ with $h(0)=0$. For such a function $h$, we denote by $\Lambda_h(E)$ the Hausdorff $h$-measure of a compact set $E$, so that the usual $s$-dimensional Hausdorff measure corresponds to $h(t)=t^s$.
\end{definition}

\begin{theorem}[Bishop \cite{BIS2}]
\label{thmBishop}
For any Hausdorff measure function $h$ with $h(t)=o(t)$ as $t \to 0$, there exists a Jordan curve $\Gamma$ such that
\begin{enumerate}[\rm(i)]
\item $\Gamma$ is not $CH$-removable;
\item $\Lambda_h(\Gamma)=0$;
\item there exists a non-M\"{o}bius $\phi \in CH(\Omega)$ with $\Lambda_h(\phi(\Gamma))=0$, where $\Omega:=\mathbb{C}_\infty \setminus \Gamma$.
\end{enumerate}
\end{theorem}

We also mention that the curve $\Gamma$ can be constructed so that it is ``highly nonremovable'', in the sense that given any other curve $\Gamma'$ and any $\epsilon>0$, there is a $\phi \in CH(\Omega)$ such that $\phi(\Gamma)$ belongs to the $\epsilon$-neighborhood of $\Gamma'$ with respect to the Hausdorff metric. Furthermore, Bishop's argument can be used to obtain an analogue of Theorem \ref{thmBishop} for totally disconnected sets.

\begin{remark}
Theorem \ref{thmBishop} implies that one can construct nonremovable curves of any Hausdorff dimension greater or equal to one. On the other hand, there are examples of compact sets of Hausdorff dimension two which are $S$-removable, hence also $CH$-removable (see \cite[Chapter V, Section 3.7]{LEH}). This shows that Corollary \ref{CorFiniteLengthCH} and Proposition \ref{RemArea} are best possible in terms of Hausdorff measures alone.
\end{remark}

We now give a sketch of the proof of Theorem \ref{thmBishop}. All the details can be found in \cite{BIS2}.

First, we introduce the following notation. If $A$ is any set, we denote by $A(\epsilon)$ the (open) $\epsilon$-neighborhood of $A$, i.e.
$$A(\epsilon):=\{z \in \mathbb{C} : \operatorname{dist}(z,A)<\epsilon\}.$$
The proof of Theorem \ref{thmBishop} is based on the following lemma on the approximation of conformal maps.

\begin{lemma}
\label{lemBishop}
Let $\Gamma$ be an analytic Jordan curve with complementary components $\Omega_1$,$\Omega_2$ and let $\psi_1,\psi_2$ be conformal maps on $\overline{\Omega_1},\overline{\Omega_2}$ such that $\psi_1(\Omega_1)$ and $\psi_2(\Omega_2)$ are the complementary components of some Jordan curve $\Gamma'$. Further, let $\alpha,\delta,\eta>0$.

Then there exists an analytic Jordan curve $\gamma \subset \Gamma(\alpha)$ with complementary components $\omega_1,\omega_2$ and conformal maps $\phi_1,\phi_2$ on $\overline{\omega_1}, \overline{\omega_2}$ such that

\begin{enumerate}[\rm(i)]
\item $\phi_1(\gamma)=\phi_2(\gamma) \subset \Gamma'(\alpha)$;
\item $|\psi_j(z)-\phi_j(z)|<\delta$ for $z \in \Omega_j \setminus \Gamma(\alpha)$, $j=1,2$;
\item $\operatorname{jump}_{\phi_1(\gamma)}(\phi_1,\phi_2)<\eta$.
\end{enumerate}
Here
$$\operatorname{jump}_{\phi_1(\gamma)}(\phi_1,\phi_2):=\sup_{x \in \gamma} \operatorname{dist}_{\phi_1(\gamma)}(\phi_1(x),\phi_2(x)),$$
where the distance is measured by arclength along $\phi_1(\gamma)=\phi_2(\gamma)$.
\end{lemma}

\begin{proof}
See \cite{BIS2}.
\end{proof}
Let us assume the above lemma and prove Theorem \ref{thmBishop}.

\begin{proof}

Let $(\epsilon_n)$ be a sequence of positive numbers decreasing to zero which we will determine later. Start with an analytic curve $\Gamma^0$ with complementary components $\Omega_1^0, \Omega_2^0$ and conformal maps $\psi_1^{0},\psi_2^{0}$ on $\overline{\Omega_1^0}, \overline{\Omega_2^0}$ mapping $\Omega_1^0, \Omega_2^0$ onto the complementary components of some Jordan curve. By Lemma \ref{lemBishop}, there exists an analytic Jordan curve $\Gamma^1 \subset \Gamma^0(\epsilon_1)$ and conformal maps $\psi_1^1,\psi_2^1$ approximating $\psi_1^0,\psi_2^0$ such that
$$\psi_1^1(\Gamma^1)=\psi_2^1(\Gamma^1) \subset \psi_1^0(\Gamma^0)(\epsilon_1) (=\psi_2^0(\Gamma^0)(\epsilon_1))$$
and
$$\operatorname{jump}(\psi_1^1,\psi_2^1)<\frac{1}{2}.$$
At the $n$-th step, we replace $\Gamma^{n-1}$ by $\Gamma^n$ and $\psi_1^{n-1},\psi_2^{n-1}$ by $\psi_1^n,\psi_2^n$ such that
$$\Gamma^n \subset \Gamma^{n-1}(\epsilon_n),$$
$$\psi_1^n(\Gamma^n)=\psi_2^n(\Gamma^n) \subset \psi_1^{n-1}(\Gamma^{n-1})(\epsilon_n) ( = \psi_2^{n-1}(\Gamma^{n-1})(\epsilon_n))$$
and
$$\operatorname{jump}(\psi_1^n,\psi_2^n)<2^{-n}.$$
Then the limits $\Gamma:=\lim_{n \to \infty} \Gamma^n$, $\psi_j:=\lim_{n \to \infty} \psi_j^n$ for $j=1,2$ exist and $\psi_1=\psi_2$ on $\Gamma$, thus these two maps define a continuous function on $\mathbb{C}_\infty$ which is conformal outside $\Gamma$. By a sufficiently small choice of each $\epsilon_n$, we can make sure that $\Gamma$ is a Jordan curve and that $\psi_1,\psi_2$ are injective on $\Gamma$, and therefore define a homeomorphism of the sphere onto itself which is conformal outside $\Gamma$. Moreover, we can arrange for this homeomorphism not to be M\"{o}bius, since $\psi_1^n,\psi_2^n$ uniformly approximate $\psi_1^0,\psi_2^0$ sufficiently far away from $\Gamma$. Finally, for $\Lambda_h(\Gamma)=\Lambda_h(\phi(\Gamma))=0$ to hold it suffices to choose $\epsilon_n$ so small that both $\Gamma^{n-1}$ and $\psi_1^{n-1}(\Gamma^{n-1})$ can be covered by $N$ disks of radius $\epsilon_n$, where $N$ is such that $Nh(\epsilon_n)$ is less than $2^{-n}$. This is possible since the curves are analytic and $h(t)=o(t)$ as $t \to 0$.

\end{proof}

\subsection{Applications to the dynamics of  complex quadratic polynomials}
In this subsection, we discuss some applications of $CH$-removability to the dynamics of quadratic polynomials. We consider the family of polynomials
$$f_c(z):=z^2+c,$$
where $c \in \mathbb{C}$.

Let us first review some elementary notions of holomorphic dynamics.

For $c \in \mathbb{C}$, the \textit{basin of infinity} of $f_c$ is defined as the set of all points that escape to infinity under iteration :
$$\mathcal{D}_c(\infty):=\{z \in \mathbb{C}_\infty : f_c^n(z) \to \infty \, \, \mbox{as} \, \, n \to \infty\},$$
where $f_c^n$ is the composition of $f_c$ with itself $n$ times. It is a completely invariant domain containing the point $\infty$. Its complement in the Riemann sphere is denoted by $\mathcal{K}_c$ and is called the \textit{filled Julia set}. The filled Julia set and the basin of infinity have a common boundary $\mathcal{J}_c:=\partial \mathcal{K}_c = \partial \mathcal{D}_c(\infty)$ called the \textit{Julia set}. The Julia set is either connected or a Cantor set (totally disconnected perfect compact set), and the latter case happens if and only if $0 \in \mathcal{D}_c(\infty)$. The \textit{Fatou set} $\mathcal{F}_c$ is defined as the complement of the Julia set :
$$\mathcal{F}_c:= \mathbb{C}_\infty \setminus \mathcal{J}_c = \mathcal{D}_c(\infty) \cup \operatorname{int}(\mathcal{K}_c).$$
It is also the maximal set of normality of the sequence of iterates $(f_c^n)_{n \in \mathbb{N}}$.

The \textit{Mandelbrot set} $M$ is the set of all parameters $c \in \mathbb{C}$ such that the Julia set $\mathcal{J}_c$ is connected. It is a connected compact set. A famous conjecture (the so-called \textit{MLC-conjecture}) asserts that $M$ is locally connected.

Let $z_0$ be a periodic point of $f_c$ of period $p$, meaning that $p$ is the smallest integer such that $f_c^p(z_0)=z_0$. We define the \textit{multiplier} of $z_0$ as
$$\lambda(z_0):= (f_c^p)'(z_0)=\prod_{n=0}^{p-1} f_c'(f_c^n(z_0)).$$
The point $z_0$ is called \textit{attracting} if $|\lambda(z_0)|<1$.

We say that a quadratic polynomial $f_c$ is \textit{hyperbolic} if it has an attracting periodic point or if its Julia set is a Cantor set. This is equivalent to the dynamics being \textit{expanding} on the Julia set. The following result on hyperbolic Julia sets is well-known.

\begin{proposition}
Let $f_c$ be a hyperbolic quadratic polynomial whose Julia set $\mathcal{J}_c$ is connected. Then the Riemann map $h:\mathbb{D} \to \mathcal{D}_\infty$ admits a H\"{o}lder continuous extension to the closed unit disk $\overline{\mathbb{D}}$. In particular, $\mathcal{J}_c$ is locally connected.
\end{proposition}
Combining this with Jones and Smirnov's result (Corollary \ref{josHolder}), we obtain
\begin{theorem}
\label{thmHyper}
If $f_c$ is a hyperbolic quadratic polynomial whose Julia set $\mathcal{J}_c$ is connected, then $\mathcal{J}_c$ is $CH$-removable.
\end{theorem}
Moreover, it is possible to show that totally disconnected Julia sets are boundaries of John domains \cite{CAR3}, thus they must be $CH$-removable by Theorem \ref{jos} or \cite{JON}. In fact, using Mcmullen's sufficient condition \cite[Theorem 2.16]{MCM}, it is not difficult to prove that these Julia sets have absolute area zero, so they are actually $S$-removable by Theorem \ref{AhlforsB}. Combining this with Theorem \ref{thmHyper}, we obtain that every homeomorphism of the sphere onto itself that is conformal outside the Julia set of a hyperbolic quadratic polynomial is actually conformal everywhere. Note that this is false if we consider rational functions instead of polynomials. Indeed, there are hyperbolic rational functions with nonremovable Julia sets; an example is given by $R(z):= z^2 + \lambda/z^3$ for $\lambda>0$ small. The Julia set of $R$ is a nonremovable Cantor set of circles of zero area. It is, however, \textit{dynamically removable}. See \cite[Section 9.2]{GRA2} and \cite[Section 11.8]{BEA}.

Other examples of $CH$-removable Julia sets include those of \textit{subhyperbolic} or \textit{Collet-Eckmann} quadratic polynomials (see \cite{JON} and \cite{GRA} for their definition), since they are known to be boundaries of John domains \cite{CAR2} and of H\"{o}lder domains \cite{GRA} respectively. It was conjectured in \cite{JON} that all Julia sets of quadratic polynomials are $CH$-removable. This is now well-known to be false, since there are Julia sets of positive area \cite{BUF}. We do not know any example though of a nonremovable quadratic Julia set with zero area.

Lastly, we mention the work of Jeremy Kahn who proved in his Ph.D. thesis that Julia sets of quadratic polynomials $f_c$ with $c \in M$ such that either
\begin{enumerate}[\rm(i)]
\item both of the fixed points of $f_c$ are repelling and $f_c$ is not renormalizable

or

\item all of the periodic cycles of $f_c$ are repelling and $f_c$ is not infinitely renormalizable
\end{enumerate}
are $CH$-removable. Then one can deduce from this that the Mandelbrot set $M$ is locally connected at such $c$'s. This illustrates the importance of studying the $CH$-removability of Julia sets of quadratic polynomials whose parameter belongs to the boundary of the Mandelbrot set. We also remark that it is not known whether the boundary of $M$ itself is $CH$-removable. It is not even known if it has zero area, although we now know thanks to a result of Shishikura \cite{SHI} that its Hausdorff dimension is two.

\subsection{Applications to Conformal welding}

In this subsection, we present another application of $CH$-removability.

Let $\mathbb{D}$ be the open unit disk, let $\mathbb{T}:=\partial \mathbb{D}$ be the unit circle and set $\mathbb{D}^{*}:=\mathbb{C}_\infty \setminus \overline{\mathbb{D}}$. Let $\Omega \subset \mathbb{C}_\infty$ be a Jordan domain and let $f:\mathbb{D} \to \Omega$ and $g:\mathbb{D}^* \to \Omega^*$ be conformal maps, where $\Omega^*:=\mathbb{C}_\infty \setminus \overline{\Omega}$. By a well-known theorem of Carath\'eodory, the maps $f$ and $g$ extend to homeomorphisms on $\overline{\mathbb{D}}$, so that $h:=f^{-1} \circ g : \mathbb{T} \to \mathbb{T}$ is an orientation-preserving homeomorphism.

\begin{definition}
We say that $h: \mathbb{T} \to \mathbb{T}$ is a \textit{conformal welding} for $\Omega$.
\end{definition}

Note that $h$ is uniquely defined up to post-composition and pre-composition with automorphisms of $\mathbb{D}$. Moreover, if $T$ is any M\"{o}bius transformation, then $\Omega$ and $T(\Omega)$ give rise to the same welding homeomorphism and thus the map
$$\mathcal{W}: [\Omega] \to [h]$$
is well-defined, where
$$[\Omega]:= \{T(\Omega) : T \, \mbox{is a M\"{o}bius transformation} \}$$
and
$$[h]:= \{ \phi \circ h \circ \psi : \phi, \psi \in \operatorname{Aut}(\mathbb{D}) \}.$$

Conformal welding has several important applications. For instance, it is a fundamental notion in the theory of Teichm\"{u}ller spaces and Fuchsian groups. It was also used by Courant in the 1930's in his solution of the Plateau-Douglas problem of minimal surfaces. More recently, it was observed by Sharon and Mumford \cite{SHM} that conformal welding is a useful tool in the field of computer vision and numerical pattern recognition, especially for the problem of classifying and recognizing objects from their observed silhouette. For more information on the various applications of conformal welding, the interested reader may consult the survey article \cite{HAM} by Hamilton and the references therein.

It is well-known that the map $\mathcal{W}$ is not surjective. However, its image contains the set of \textit{quasisymmetric} homeomorphisms; this is usually referred to as the \textit{fundamental theorem of conformal welding} and it was first proved by Pfluger \cite{PFL} in 1960. Another proof was published shortly after by Lehto and Virtanen \cite{LEH2}. We also refer the reader to \cite{BIS3} for an elementary geometric proof using Koebe's circle domain theorem as well as \cite{SCH} for a functional analytic proof.

As for the injectivity of the map $\mathcal{W}$, the following proposition shows that it is closely related to removability properties.

\begin{proposition}
\label{Uniqueness}
Let $\Omega, \tilde{\Omega} \subset \mathbb{C}_\infty$ be Jordan domains. Then $\mathcal{W}([\Omega])=\mathcal{W}([\tilde{\Omega}])$ if and only if there exists a map $F\in CH(\mathbb{C}_\infty \setminus \partial \Omega)$ such that $F(\Omega)=\tilde{\Omega}$.
\end{proposition}

\begin{proof}
If $\mathcal{W}([\Omega])=\mathcal{W}([\tilde{\Omega}])$, then there exist conformal maps $f:\mathbb{D} \to \Omega$, $g:\mathbb{D}^* \to \Omega^*$, $\tilde{f}:\mathbb{D} \to \tilde{\Omega}$, $\tilde{g}:\mathbb{D}^* \to \tilde{\Omega}^*$ such that
$$f^{-1} \circ g = \tilde{f}^{-1} \circ \tilde{g}$$
on $\mathbb{T}$, i.e.
$$\tilde{f} \circ f^{-1}=\tilde{g} \circ g^{-1}$$
on $\partial \Omega$. It follows that the conformal map $\tilde{f} \circ f^{-1}$ on $\Omega$ can be extended to a homeomorphism $F:\mathbb{C}_\infty \to \mathbb{C}_\infty$ which is conformal outside $\partial \Omega$. Clearly, $F(\Omega)=\tilde{\Omega}$.

Conversely, if there exists such an $F$, then $F \circ f: \mathbb{D} \to \tilde{\Omega}$ and $F \circ g : \mathbb{D}^* \to \tilde{\Omega}^*$ are conformal whenever $f:\mathbb{D} \to \Omega$ and $g:\mathbb{D}^* \to \Omega^*$ are conformal, so that
$$\mathcal{W}([\tilde{\Omega}]) = [(F \circ f)^{-1} \circ (F \circ g)] = [f^{-1} \circ g] = \mathcal{W}([\Omega]).$$
\end{proof}

\begin{corollary}
\label{UniquenessWelding}
If $\Omega$ is a Jordan domain such that $\partial \Omega$ is $CH$-removable, then for any other Jordan domain $\tilde{\Omega}$, we have $\mathcal{W}[\Omega]=\mathcal{W}[\tilde{\Omega}]$ if and only if $[\Omega]=[\tilde{\Omega}]$.
\end{corollary}

In other words, if $\partial \Omega$ is $CH$-removable, then the only other Jordan domains that yield the same welding homeomorphism are the M\"{o}bius images of $\Omega$. It was observed by Maxime Fortier Bourque in his master's thesis that many authors have claimed that the converse is true, either without proof or by giving an incorrect argument. See for instance \cite[Lemma 2]{OIK}, \cite[Corollary II.2]{KNS}, \cite[Section 4]{HAM2} \cite[Corollary 1]{BIS}, \cite[p.324--325]{BIS2}, \cite[Section 3]{HAM}, \cite[Remark 2]{BIS3}, \cite[Section 2.3]{AST2}, \cite[Corollary 1.4]{LIR}. If $\partial \Omega$ is not $CH$-removable, then there exists a non-M\"{o}bius homeomorphism $F: \mathbb{C}_\infty \to \mathbb{C}_\infty$ which is conformal outside $\partial \Omega$. As in Proposition \ref{Uniqueness}, it follows that $\Omega$ and $F(\Omega)$ give rise to the same welding homeomorphism. However, one cannot directly deduce that $[F(\Omega)] \neq [\Omega]$ since there could exist a M\"{o}bius transformation $T$ such that $F(\Omega)=T(\Omega)$, even though $F$ itself is non-M\"{o}bius.

However, if in addition the boundary of $\Omega$ is assumed to have positive area, then the converse of Corollary \ref{UniquenessWelding} holds. Indeed, by the measurable Riemann mapping theorem, there is an infinite-dimensional family of non-M\"{o}bius homeomorphisms of the sphere conformal outside $\partial \Omega$ and one can use some dimension argument to show that the images of $\Omega$ under these conformal homeomorphisms cannot always be M\"{o}bius-equivalent to $\Omega$. A similar argument was used by Sullivan in his proof of the No Wandering Domain theorem.

Unfortunately, as we saw earlier, there are curves of zero area that are not $CH$-removable. We do not know if the converse of Corollary \ref{UniquenessWelding} holds for such curves.
\begin{question}
\label{Quniqueness}
Is the converse of Corollary \ref{UniquenessWelding} true in the zero-area case? In other words, if $\Omega$ is a Jordan domain with zero area boundary such that the only other Jordan domains giving rise to the same welding homeomorphism are the M\"{o}bius images of $\Omega$, then is $\partial \Omega$ necessarily $CH$-removable?
\end{question}

Finally, we end this subsection by mentioning that in some special cases, the welding homeomorphism can be identified explicitly. For instance, the welding homeomorphism of a \textit{proper polynomial lemniscate of degree} $n$ (i.e. a connected set of the form $\Omega:=\{z:|P(z)|<1\}$ where $P$ is a polynomial of degree $n$) is given by the $n$-th root of a Blaschke product of degree $n$ whose zeros are the images of the zeros of $P$ under a conformal map of $\Omega$ onto $\mathbb{D}$. Conversely, any $n$-th root of a Blaschke product of degree $n$ is the welding homeomorphism of a proper polynomial lemniscate of the same degree. This was first proved by Ebenfelt, Khavinson and Shapiro \cite{EKS}. See also \cite{YOU2} for a more elementary proof and a generalization to rational lemniscates.

There are also several efficient numerical methods that can be used to recover the boundary curve $\partial \Omega$ from the welding homeomorphism, such as Marshall's \textit{Geodesic Zipper Algorithm} \cite{MAR} for example.

\subsection{Open questions}
We end this section by discussing some open questions related to $CH$-removability.

\begin{question}
\label{Q1}
Is the union of two $CH$-removable compact sets also $CH$-removable?
\end{question}
Recall that the answer is yes if the two compact sets are assumed to be disjoint, by Corollary \ref{QCHunion}. Furthermore, by Proposition \ref{union}, Proposition \ref{unionA} and Proposition \ref{unionS}, the answer is positive if the class $CH$ is replaced by $H^{\infty}, A$ or $S$.

The main difficulty in Question \ref{Q1} is showing that Proposition \ref{PropExtendQCH} remains true without the assumption that $E$ is contained in $U$ in (i). For the classes $H^{\infty}$ and $A$, we were able to prove the corresponding result using Vitushkin's localization operator and for the class $S$, using the fact that $S$-removable sets are totally disconnected. Unfortunately, these two approaches seem to fail for $CH$-removable sets since they are not necessarily totally disconnected and since Vitushkin's localization operator does not preserve injectivity.

\begin{question}[Bishop \cite{BIS2}]
\label{Q2}
Let $\Gamma$ be a Jordan curve. If $\Gamma$ is not $CH$-removable, does it contain a nonremovable closed proper subset?
\end{question}
Clearly, the answer is positive if $\Gamma$ has positive area. A positive answer in the general case would obviously follow from a positive answer to Question \ref{Q1}.

\begin{question}
\label{Q3}
Let $E \subset \mathbb{C}$ be a compact set which is not $CH$-removable. How large is $CH(\Omega)$? In particular, if $f \in CH(\Omega)$ is non-M\"{o}bius, does there exist another non-M\"{o}bius homeomorphism in $CH(\Omega)$ which is not of the form $T \circ f$ where $T$ is M\"{o}bius?
\end{question}
As mentioned before Question \ref{Quniqueness}, if $E$ has positive area then $CH(\Omega)$ is very large. An affirmative answer to Question \ref{Quniqueness} would follow if one could prove that $CH(\Omega)$ is always large enough, even in the zero-area case.

Finally, recall that if $E$ is not removable for the class $S$, then there exists a conformal map of $\Omega$ onto the complement of a set of positive area, in view of Theorem \ref{AhlforsB}.
\begin{question}
\label{Q4}
If $E\subset \mathbb{C}$ is compact and not $CH$-removable, does there exist a map $f \in CH(\Omega)$ such that $f(E)$ has positive area?
\end{question}
If we could prove that the answer is yes, then this would give a positive answer to Question \ref{Quniqueness}.

\acknowledgments{The author would like to thank Chris Bishop and Misha Lyubich for their interest and helpful discussions, especially on conformal removability and its applications in holomorphic dynamics. We also thank Joe Adams for reading the first draft of the paper and providing valuable comments.}

\bibliographystyle{amsplain}

\end{document}